\documentclass{amsart}
\usepackage{color}
\usepackage{amsmath,amssymb}

\newtheorem{theorem}{Theorem}[section]
\newtheorem{proposition}[theorem]{Proposition}
\newtheorem{corollary}[theorem]{Corollary}
\newtheorem{lemma}[theorem]{Lemma}

\theoremstyle{definition}

\newtheorem{example}[theorem]{Example}

\theoremstyle{remark}

\numberwithin{equation}{section}

\def\N{\mathbb N}
\newcommand\hdim{\dim_{\mathrm H}}
\usepackage{colortbl}

\begin{document}

\title[set approximated by irrational rotations]{Hausdorff dimension of the set approximated by irrational rotations}

\author{Dong Han Kim}
\address{Department of Mathematics Education, Dongguk University - Seoul, Seoul, 04620 Korea}
\email{kim2010@dongguk.edu}

\author{Micha\l{} Rams}
\address{Institute of Mathematics, Polish Academy of Sciences, 00-956 Warszawa, Poland}
\email{rams@impan.pl}

\author{Baowei Wang}
\address{School of Mathematics and Statistics, Huazhong University of Science and Technology, Wuhan, 430074, China}
\email{bwei\_wang@hust.edu.cn}

\begin{abstract}
Let $\theta$ be an irrational number and $\varphi: {\mathbb N} \to {\mathbb R}^{+}$ be a monotone decreasing function tending to zero.
Let
$$E_\varphi(\theta) =\Big\{y \in \mathbb R: \|n\theta- y\|<\varphi(n), \ {\text{for infinitely many}}\ n\in {\mathbb N} \Big\}, $$
i.e. the set of points which are approximated by the irrational rotation with respect to the error function $\varphi(n)$.
In this article,  we give a complete description of the Hausdorff dimension of $E_\varphi(\theta)$ for any monotone function $\varphi$ and any irrational $\theta$.
\end{abstract}

\keywords {Inhomogeneous Diophantine approximation; Shrinking target problem; Irrational rotation; Hausdorff dimension.}

\subjclass[2010]{Primary 11K55, Secondary 11J71, 28A80}

\maketitle

\section{Introduction}

Let $\theta$ be an irrational number.
The distribution of the sequence $\{n\theta\}_{n\ge 1}$ in $\mathbb R/ \mathbb Z$ is a classical topic in Diophantine approximation.
It has been studied that $\{n\theta\}_{n\ge 1}$ is well distributed.
Weyl's uniform distribution theorem (e.g., \cite{KN}) states that $\{n\theta\}_{n\ge 1}$ is uniformly distributed in $\mathbb R/\mathbb Z$ in the sense that  for any interval $(a,b) \subset [0,1)$,
$$
\lim_{N \to \infty}\frac{1}{N} \# \left\{1\le n \le N:  \langle n \theta \rangle \in (a,b) \right\} = b-a,
$$
where $\langle t \rangle$ for the fractional part of a real number $t$.
Minkowski's theorem (e.g. \cite{Min57}) states that if $y$ is not of the form $m+\ell \theta$ for integers $m, \ell$, then there exist infinitely many integers $n$ such that
$$\|n\theta-y\|< \frac{1}{4|n|}$$
where $\|\cdot\|$ for the distance to the nearest integer.
Moreover, it was shown by Schmeling and Troubetzkoy \cite{ScTr} and by Bugeaud \cite{Bu} independently (see also \cite{BV}) that for $\gamma \ge 1$
\begin{equation}\label{f1}
\hdim \left \{y\in \mathbb R : \|n\theta-y\|<n^{-\gamma} \ \text{ for infinitely many } n\in \N\right\} = \frac1\gamma,
\end{equation}
where $\hdim$ denotes the Hausdorff dimension.
These results are all not dependent on the irrational $\theta$, but if we replace $1/n^{\gamma}$ with a general monotone function $\varphi(n)$, the situation becomes quite different.

Suppose that $\varphi(n)$ is a positive monotone decreasing function. Let
$$E_\varphi(\theta) =\left \{y \in \mathbb R: \|n\theta- y\|<\varphi(n) \ \text{ for infinitely many } n\in {\mathbb N} \right\}. $$
For the Lebesgue measure of $E_\varphi(\theta)$,
Fuchs and Kim \cite{FK} showed that
\begin{equation}\label{FK}
\mathrm{Leb}(E_\varphi(\theta))=1 \ \text{ if and only if } \  \sum_{k=0}^{\infty}\sum_{n=q_k}^{q_{k+1}-1}\min\{\varphi(n), \|q_k\theta\|\}=\infty,
\end{equation}
where $\{q_k\}_{k\ge 1}$ are the denominators of the principal convergents $p_k/q_k$ of $\theta$. (For earlier works, see \cite{BD,Kim,Kurz,Tseng}.)
This shows that the size of $E_\varphi(\theta)$ depends heavily on the Diophantine properties of $\theta$.

For the Hausdorff dimension, Fan and Wu \cite{FanWu} presented an example indicating that $\hdim E_\varphi(\theta)$ also depends heavily on $\theta$. Put
$$
u_\varphi := \limsup_{n \to \infty} \frac{\log n}{-\log \varphi(n)}, \qquad l_\varphi := \liminf_{n \to \infty} \frac{\log n}{- \log \varphi(n)}.
$$
Then, Xu \cite{Xu} showed that
\begin{equation}\label{eq:Xu}
\limsup_{n \to \infty} \frac{\log q_k}{-\log \varphi(q_k)} \le \hdim(E_\varphi(\theta)) \le u_\varphi.
\end{equation}
Moreover, Liao and Rams \cite{LR} proved that
\begin{equation}\label{eq:LR}
\min \left\{ u_\varphi, \max \Big\{ l_\varphi, \frac{1+u_\varphi}{1+w} \Big\} \right\} \le \hdim(E_\varphi(\theta)) \le u_\varphi,
\end{equation}
where
$$ w = \limsup_{k \to \infty} \frac{\log q_{k+1}}{\log q_k}.$$

Our main theorem gives the full description of the Hausdorff dimension of $E_\varphi(\theta)$:

\begin{theorem}\label{thm}
For $0\le s\le 1$, let $q_{k,s}$ be the integer $m$ with $q_k \le m < q_{k+1}$ minimizing
\begin{equation}\label{qks}
q_k \left( \varphi(q_k) + \frac{m -q_k}{q_k}\|q_{k}\theta\| \right)^s + \sum_{n = m +1}^{q_{k+1}-1} \varphi(n)^s.
\end{equation}
Then we have
\begin{multline*}
\hdim (E_\varphi(\theta) ) \\
= \inf \left \{ s > 0 : \sum_{k=0}^\infty \left[ q_k \left( \varphi(q_k) + \frac{q_{k,s}-q_k}{q_k}\|q_{k}\theta\| \right)^s + \sum_{n = q_{k,s}+1 }^{q_{k+1}-1} \varphi(n)^s \right] < \infty \right \}.
\end{multline*}
\end{theorem}

Though this statement is complicated, easy calculations lead to the earlier results:
we deduce \eqref{eq:LR} by Liao and Rams and \eqref{eq:Xu} by Xu from Theorem~\ref{thm} in Section~\ref{sec:relations}.
In Section~\ref{sec:MTP},  we discuss the lower bound of the Hausdorff dimension by the mass transference principle from \cite{BV}.
The proof of the main theorem is given in Section~\ref{sec:proof}.

The proof of the main theorem strongly depends on the result of Liao and Rams.
The following proposition, which is going to be crucial in our work, is Proposition 2.3 of \cite{LR}:

\begin{proposition}[\cite{LR}]\label{prop:LR}
Suppose $K > 1$ and $N >1$.
Let $\{k_i\}$ be a sufficiently fast increasing sequence of integers, for which $\frac{\log q_{k_i+1}}{\log q_{k_i}} \to B$. 
Let $\{m_i\}$ be a sequence of reals satisfying $q_{k_i} \le m_i = q_{k_i}^N < q_{k_i+1}$. 
Denote
$$F_i := \left\{ y : \| n\alpha - y \| < \frac 1{2q_{k_i}^{K}} \text{ for some } m_{i-1} < n \le m_i \right\}. $$
Then
$$\hdim \left(\bigcap_{i =1}^\infty F_i \right) =\min \left\{\frac{N}{K}, \, \max \Big\{\frac{1}{K}, \frac{1}{1+B-N}\Big\}\right\}. $$
\end{proposition}

We will also need a version of this result which was not proven in \cite{LR}, but the proof of which is basically identical.

\begin{proposition}\label{prop:LR2}
Let $\{k_i\}$ be a sufficiently fast increasing sequence of integers. 
Let $\{K_i\}$ and $\{m_i\}$ be two sequences of positive reals satisfying $K_i > 1$ and $q_{k_i} \le m_i: = q_{k_i}^{N_i} < q_{k_i+1}$. 
Write $\frac{\log q_{k_i+1}}{\log q_{k_i}}=B_i$. 
Assume that $K_i$, $B_i-N_i$ and $\frac{K_i}{N_i}$ converge to $K,  D,  R$ respectively.
Define
$$F_i := \left\{ y : \| n\alpha - y \| < \frac 1{2q_{k_i}^{K_i}} \text{ for some } m_{i-1} < n \le m_i \right\}. $$
Then
$$\hdim \left(\bigcap_{i =1}^\infty F_i \right) = \min \left\{\frac{1}{R}, \, \max \Big\{\frac{1}{K}, \frac{1}{1+D}\Big\}\right\}.$$
\end{proposition}

We remark that the exact growth rate of $\{k_i\}_{i\ge 1}$ necessary for Propositions \ref{prop:LR} and \ref{prop:LR2} is not going to be important for us.

\section{Mass transference principle}\label{sec:MTP}

The mass transference principle, developed by Beresnevich and Velani \cite{BV}, is a
powerful tool to determine the Hausdorff measure and dimension of a limsup set.
For example, equipped with the Minkowski's result, the dimensional result of (\ref{f1}) can be obtained by using the mass transference principle directly.
Let $B(x_i, r_i)$ be a sequence of balls in $[0,1)$ with centers $x_i$'s and radii $r_i$'s with $r_i \to 0$.
If $ \limsup B(x_i, r_i^s )$ has the full Lebesgue measure for $0 < s \le1$, then  
it follows by the mass transference principle \cite{BV} that
$$ \hdim \Big( \limsup B(x_i, r_i) \Big) \ge s. $$
Combining \eqref{FK} with the mass transference principle,
we have
\begin{equation}\label{ineq:MT}
\hdim (E_\varphi(\theta) ) \ge \inf\left\{s\ge 0: \sum_{k=0}^{\infty} \sum_{n=q_k}^{q_{k+1}-1} \min\Big\{ \varphi(n)^s, \|q_k\theta\|\Big\}<\infty\right\}.
\end{equation}
However,
this may not be the exact dimension of $E_\varphi(\theta)$.

For each $0 \le s \le 1$, we have
\begin{equation*}\begin{split}
\sum_{n=q_k}^{q_{k+1}-1} \min \Big\{ \varphi(n)^s, \|q_k\theta\|\Big\}
&\le \varphi(q_k) + (q_{k,s}-q_k)\|q_{k}\theta\| + \sum_{n = q_{k,s}+1 }^{q_{k+1}-1} \varphi(n)^s\\
&\le q_k \left( \varphi(q_k) + \frac{q_{k,s}-q_k}{q_k}\|q_{k}\theta\| \right) + \sum_{n = q_{k,s}+1 }^{q_{k+1}-1} \varphi(n)^s\\
&\le q_k \left( \varphi(q_k) + \frac{q_{k,s}-q_k}{q_k}\|q_{k}\theta\| \right)^s + \sum_{n = q_{k,s}+1 }^{q_{k+1}-1} \varphi(n)^s.
\end{split}
\end{equation*}
Therefore, the lower bound given by \eqref{ineq:MT} is less than or equal to the Hausdorff dimension given by Theorem~\ref{thm}.
However the following example shows that the lower bound may not be sharp.

\begin{example}
Choose an irrational $\theta$ given by the partial quotients $a_{k+1} = q_k^{2}$
and the error function $\varphi(n) = 1/2q_{k}^{3}$ for each $n \in (q_{k-1}^2, q_k^2]$.
By elementary calculus we get $q_{k,s} = q_k^2$. Thus,
\begin{align*}
q_k \left( \varphi(q_k) + \frac{q_{k,s}-q_k}{q_k}\|q_{k}\theta\| \right)^s + \sum_{n = q_{k,s}+1 }^{q_{k+1}-1} \varphi(n)^s
&< q_k^{s+1} \| q_k \theta\|^s + \frac{q_{k+1}}{2^sq_{k+1}^{3s}} \\
&< \frac{1}{q_k^{2s-1}} + \frac{1}{2^sq_{k+1}^{3s-1}},
\end{align*}
which converges for $s > 1/2$.
We also have
\begin{align*}
q_k \left( \varphi(q_k) + \frac{q_{k,s}-q_k}{q_k}\|q_{k}\theta\| \right)^s + \sum_{n = q_{k,s}+1 }^{q_{k+1}-1} \varphi(n)^s
&> q_k (q_k -1 )^s \| q_k \theta\|^s > \frac{q_k^{1+s}}{2^{3s} q_k^{3s}},
\end{align*}
which diverges for $s <1/2$.
Therefore, the Hausdorff dimension of $E_\varphi(\theta)$ is 1/2.

On the other hand, for $1/3 < s <1$,
we have
$$\sum_{q_k\le n<q_{k+1}} \min\Big\{ \varphi(n)^s, \|q_k\theta\|\Big\}
= (q_k^2 -q_k +1)\| q_k \theta\| + \frac{q_{k+1}-q^2_k -1}{2q_{k+1}^{3s}}
<\frac{1}{q_k} + \frac{1}{2q_{k+1}^{3s-1}},
$$
which is a convergent series.
If $0 < s <1/3$, then for large $k$ so that $a_{k+1} = q_k^2 > 3$
$$\sum_{q_k\le n<q_{k+1}} \min\Big\{\varphi(n)^s,  \|q_k\theta\|\Big\}
=(q_{k+1} - q_k) \| q_k\theta\| > \frac{q_{k+1} - q_k}{q_{k+1} + q_k} > \frac 12, $$
which is a divergent series.
Hence, the lower bound of the Hausdorff dimension given by \eqref{ineq:MT} is 1/3.
\end{example}

\section{Proof of the main theorem}\label{sec:proof}

Let $\theta=[a_1,a_2,\cdots]$ be the continued fraction expansion of $\theta$ and $\{q_k\}$ its denominators of the convergents, which satisfy $q_{k+1}=a_{k+1}q_k+q_{k-1}$.
Since $q_{k+1} \ge 2 q_{k-1}$, it increases exponentially. Thus 
for any $\varepsilon >0$
$$\sum_{k=0}^\infty \frac{1}{q_k^\varepsilon} < \infty.$$
We will use this convergence implicitly.

The Three Distances Theorem (e.g., \cite{Ha}) plays a central role in the proof.
 It states that for each positive interger $N$, the gaps between two consecutive points of $\langle\theta\rangle, \langle2\theta\rangle , \dots ,\langle N\theta \rangle$ have at most three lengths if the points are arranged in ascending order.
In the special case, 
when $N= q_{k+1}$ the gap between two neighboring points in $\{  \langle \theta\rangle, \langle2\theta\rangle, \dots, \langle q_{k+1} \theta \rangle \}$ is $\|q_{k}\theta\|$ or $\|q_{k}\theta\| + \|q_{k+1}\theta\|$.
To simplify the notation, we view $n\theta$ as a number on the unit circle $S^1 = \mathbb R/ \mathbb Z$.

\begin{proof}[Proof of the upper bound]
To give an
upper bound for the Hausdorff dimension of $E_\varphi(\theta)$, we need to find a suitable cover.
We will start by covering not the set $E_\varphi(\theta)$ itself (which is dense in $S^1$) but rather the set
$$
E_k=\bigcup_{q_k\le n<q_{k+1}}B(n\theta, \varphi(n))
$$
for each $k\ge 1$.
Since $E_{\varphi}(\theta)=\limsup_{k\to \infty}E_k$, the covers of $E_k$ let us construct a cover of $E_{\varphi}(\theta)$.
Before the main argument, we explain the idea of looking for an optimal cover of $E_k$.

At first, we arrange
 the balls $B(n\theta, \varphi(n))$ in $E_k$ into $q_k$ groups according to the positions of $n\theta$, $q_k\le n<q_{k+1}$:
\begin{align}\label{group} 
                \begin{array}{ll}
                  \big\{B(n\theta, \varphi(n)): n=tq_k+i, \ 1\le t\le a_{k+1}\big\}, & \hbox{when $0\le i<q_{k-1}$;} \\
                  \big\{B(n\theta, \varphi(n)): n=tq_k+i, \ 1\le t<a_{k+1}\big\}, & \hbox{when $q_{k-1}\le i<q_{k}$.}
                \end{array}
\end{align}

Inside each group, it may happen that some balls are overlapping or otherwise sufficiently close so that it will be more efficient to cover several of them with one interval. So, we cover the balls close enough by a long interval, while for other well separated balls, we cover them by themselves. Note that the distances between the centers of neighboring balls are equal (Three Distances Theorem) while their diameters shrink (monotonicity of $\varphi$), hence we get the following picture: one long interval followed by a sequence of short intervals. This gives a cover of $E_k$. The role of $q_{k,s}$ plays is just to optimize the summation of the $s$-volume of such a cover.

Now let's give the detailed argument. Let $0 < s \le 1$ be a real number satisfying that
$$ \sum_{k=0}^\infty \left[ q_k \left( \varphi(q_k) + \frac{q_{k,s} -q_k}{q_k}\|q_{k}\theta\| \right)^s + \sum_{n = q_{k,s}+1}^{q_{k+1}-1} \varphi(n)^s \right] < \infty . $$

For each $k\ge 1$, inside each group given in (\ref{group}), we cover the balls $\{B(n\theta, \varphi(n))\}$ with $n\le q_{k,s}$ by a long interval.
More precisely, let $c$, $r$ be the integers such that
$q_{k,s} = c q_k + r$ with $c \ge 1$, $0 \le r \le q_k -1$.
Recall that $\| q_k \theta \|  = (-1)^k(q_k \theta - p_k)$. 

If $c \ge2$, then let for even $k$
\begin{equation*}
C_{k,i} :=
\begin{cases}
\Big( (q_k + i) \theta - \varphi(q_k),  ( c q_k + i) \theta + \varphi(q_k) \Big)  & \text{ if } 0 \le i \le r, \\
\Big( (q_k + i) \theta - \varphi(q_k),  ( (c-1) q_k + i) \theta + \varphi(q_k) \Big)& \text{ if } r < i < q_k.
\end{cases}
\end{equation*}
and for odd $k$
\begin{equation*}
C_{k,i} :=
\begin{cases}
\Big( ( cq_k + i) \theta - \varphi(q_k), (q_k+i) \theta + \varphi(q_k) \Big) & \text{ if } 0 \le i \le r, \\
\Big( ( (c-1)q_k + i) \theta - \varphi(q_k), (q_k + i) \theta + \varphi(q_k) \Big)& \text{ if } r < i < q_k.
\end{cases}
\end{equation*}
If $c=1$, then for $0 \le i \le r$ let
\begin{equation*}
C_{k,i} :=
\begin{cases}
\Big( (q_k + i) \theta - \varphi(q_k), (q_k + i) \theta + \varphi(q_k) \Big), &\text{ for even } k,  \\
\Big( ( q_k + i) \theta - \varphi(q_k), (q_k+i) \theta + \varphi(q_k) \Big), &\text{ for odd } k  .
\end{cases}
\end{equation*}
and for $r < i < q_k$ let $ C_{k,i} :=  \emptyset .$
Then we have
\begin{equation*}
\bigcup_{q_k \le n \le q_{k,s}} B\left (n\theta, \varphi(n) \right)
\subset \bigcup_{q_k \le n \le q_{k,s}} B\left (n\theta, \varphi(q_k) \right)
\subset  \bigcup_{0 \le i < q_k} C_{k,i}.
\end{equation*}
Since
\begin{multline*}
(r+1) (2\varphi(q_k) + (c -1)\| q_k \theta\| )^s + (q_k - r -1) (2\varphi(q_k) + (c -2)\| q_k \theta\| )^s \\
\le q_k (2\varphi(q_k) + (c-1)\| q_k \theta\| )^s
< 2 q_k \left( \varphi(q_k) + \left(\frac{q_{k,s}}{q_k} -1 \right) \|q_{k}\theta\| \right)^s,
\end{multline*}
we find a covering of the set
\begin{equation*}
E_k = \bigcup_{q_k \le n < q_{k+1}} B\left (n\theta, \varphi(n) \right)
\subset \left(  \bigcup_{0 \le i < q_k} C_{k,i} \right) \cup \left( \bigcup_{q_{k,s} < n < q_{k+1}} B\left (n \theta, \varphi(n) \right) \right)
\end{equation*}
such that the sum of $s$-th powers of the lengths of covering intervals is bounded by
$$
2q_k \left( \varphi(q_k) +\frac{q_{k,s}-q_k}{q_k} \|q_{k}\theta\| \right)^s + 2\sum_{n = q_{k,s}+1}^{q_{k+1}-1} \varphi(n)^s. $$
As the union over covers of $E_k$ is a cover of $E_\varphi(\theta)$,
it implies that $\hdim(E_\varphi(\theta)) \le s$.
\end{proof}

In order to prove the lower bound, we use a lemma which is a direct consequence of  Proposition~\ref{prop:LR2}.
\begin{lemma}\label{lem:LR3}
Let $\{k_i\}$ be a sufficiently fast increasing sequence of integers. 
Let $\{K_i\}$ and $\{m_i\}$ be two sequences of positive reals satisfying $K_i > 1$ and $q_{k_i} \le m_i: = q_{k_i}^{N_i} < q_{k_i+1}$. 
Write $\frac{\log q_{k_i+1}}{\log q_{k_i}}=B_i$. 
Assume that $1+B_i-N_i \le R$ and $\frac{K_i}{N_i} \le R$.
Then
$$\hdim \left(\bigcap_{i =1}^\infty  \Big\{ y : \| n\alpha - y \| < \frac 1{2q_{k_i}^{K_i}} \text{ for some } m_{i-1} < n \le m_i \Big\} \right)  \ge \frac{1}{R}.$$
\end{lemma}

\begin{proof}[Proof of the lower bound]
Suppose that for some $0 < s < 1$
$$ \sum_{k=0}^\infty \left[ q_k \left( \varphi(q_k) + \frac{q_{k,s} -q_k}{q_k}\|q_{k}\theta\| \right)^s + \sum_{n = q_{k,s}+1}^{q_{k+1}-1} \varphi(n)^s \right] = \infty . $$
We have to show that $\hdim (E_\varphi(\theta)) \ge s -\varepsilon$ for any $\varepsilon >0$.

If  
\begin{equation*} 
\sum_{k=0}^\infty  q_k \left( \varphi(q_k) + \frac{q_{k,s}-q_k}{q_k}\|q_{k}\theta\| \right)^s = \infty,
\end{equation*}
then either 
\begin{equation}\tag{i} 
\sum_{k=0}^\infty  q_k \varphi(q_k)^s = \infty
\end{equation}
or
\begin{equation}\tag{ii}
\sum_{k=0}^\infty  q_k \varphi(q_k)^s < \infty \ \text{ and } \
 \sum_{k=0}^\infty  q_k \left( \frac{q_{k,s} -q_k}{q_k}\|q_{k}\theta\| \right)^s = \infty.
\end{equation}
Otherwise, 
\begin{equation}\tag{iii} 
\sum_{k=0}^\infty  q_k \left( \varphi(q_k) + \frac{q_{k,s}-q_k}{q_k}\|q_{k}\theta\| \right)^s < \infty  \ \text{ and }  \ \sum_{k=0}^\infty \left( \sum_{n = q_{k,s}+1}^{q_{k+1}-1} \varphi(n)^s \right) = \infty.
\end{equation}

\bigskip

Proof of Case (i) : 

\smallskip

From the condition of Case (i) 
for each $\varepsilon >0$, there exists a subsequence $\{k_i\}$ such that
$$
\varphi(q_{k_i}) > \left( \frac{1}{q_{k_i}}\right)^{\frac{1}{s-\varepsilon}}.
$$
We can freely assume that this subsequence $\{k_i\}$ grows fast enough for us to apply Proposition~\ref{prop:LR2} later on.
Put $m_i = q_{k_i}$. Then
$$\varphi(m_i) > \left( \frac{1}{q_{k_i}}\right)^{\frac{1}{s-\varepsilon}} \ge \frac{1}{2q_{k_i}^K} \quad \text{with}\  K=\frac{1}{s-\varepsilon}.$$
Let
$$E_i := \left\{ y : \| n\alpha - y \| < \frac 1{2q_{k_i}^K} \text{ for some } m_{i-1} < n \le m_i \right\}. $$
Then the monotonicity of $\varphi$ gives that
$$ E := \bigcap_{i =1}^\infty E_i \subset E_\varphi(\theta)$$
and by Proposition~\ref{prop:LR2}
$$\hdim (E) \ge \frac{1}{K} = s -\varepsilon.
$$

\bigskip

Proof of Case (ii) :  

\smallskip

Claim : For each $\varepsilon >0$, there exists a subsequence $\{k_i\}$ such that
\begin{equation}\label{casei}
\frac{q_{k_i,s}-q_{k_i}}{q_{k_i}}\|q_{k_i}\theta\| > \max\left\{\Big( \frac{1}{q_{k_i}} \Big)^{1/(s-\varepsilon)},\  \varphi(q_{k_i})\right\}.
\end{equation}

\begin{proof}[Proof of the Claim]
Suppose not. Then for all $k$ sufficiently large
$$ \frac{q_{k,s}-q_k}{q_{k}}\|q_{k}\theta\| \le \max\left\{\Big( \frac{1}{q_{k}} \Big)^{1/(s-\varepsilon)},\  \varphi(q_{k})\right\}.$$
Thus
\begin{align*}
q_{k} \left( \frac{q_{k,s}-q_k}{q_{k}}\|q_{k}\theta\| \right)^{s}
\le \max \left\{ \left(\frac{1}{q_k}\right)^{\varepsilon/(s-\varepsilon)}, \ q_k \varphi(q_{k})^s \right\},
\end{align*}
hence
\begin{equation*}
\sum_{k=0}^\infty q_{k} \left( \frac{q_{k,s}-q_k}{q_{k}}\|q_{k}\theta\| \right)^{s} < \infty. \qedhere
\end{equation*}
\end{proof}

Let $m_i := q_{k_i,s} $ and define $N_i, B_i, K_i$ so that 
$$q_{k_i}^{N_i} = m_i, \qquad  \frac{1}{q_{k_i}^{B_i}} = \| q_{k_i} \theta \|, \qquad  \frac{1}{q_{k_i}^{K_i}}= \varphi(m_i).$$
Then by the claim
$$
 \frac{1}{q_{k_i}^{K_i}}= \varphi(m_i) \le \varphi(q_{k_i}) < \frac{q_{k_i,s}-q_{k_i}}{q_{k_i}}\|q_{k_i}\theta\|  < \frac{q_{k_i +1}\|q_{k_i}\theta\| }{q_{k_i}} < \frac{1}{q_{k_i}}
$$
and
$$ q_{k_i}^{1 + (s-\varepsilon) (N_i-1-B_i)} = q_{k_i} \left( \frac{q_{k_i,s}}{q_{k_i}}\|q_{k_i}\theta\| \right)^{s-\varepsilon} > q_{k_i} \left( \frac{q_{k_i,s}-q_{k_i}}{q_{k_i}}\|q_{k_i}\theta\| \right)^{s-\varepsilon} >1 , $$
which implies
$$
K_i > 1 
$$
and
\begin{equation}\label{ineq3}
1+B_i - N_i < \frac{1}{s - \varepsilon}. 
\end{equation}

Set
$$F := \bigcap_{i=1}^\infty F_i,$$
where
$$F_i = \left \{ y : \| n\alpha - y \| < \frac{1}{2q_{k_i}^{K_i}}  \text{ for some } m_{i-1} < n \le m_i \right \}. $$
Clearly, $$F \subset E_\varphi(\theta).$$

If $q_{k,s} \ge q_k +1$, then 
by the minimality condition of $q_{k,s}$, we have  
$$ q_k \left( \varphi(q_k) + \frac{q_{k,s} - q_k}{q_k} \| q_k \theta \| \right)^s
\le q_k \left( \varphi(q_k) + \frac{(q_{k,s}- 1)- q_k}{q_k} \| q_k \theta \| \right)^s + \varphi(q_{k,s})^s,  $$  
Therefore,  
\begin{equation*}
\begin{split}
\varphi(q_{k,s})^s &\ge q_k \left( \varphi(q_k) + \frac{q_{k,s} - q_k}{q_k} \| q_k \theta \| \right)^s
- q_k \left( \varphi(q_k) + \frac{q_{k,s}- q_k-1}{q_k} \| q_k \theta \| \right)^s \\
&= q_k \left( \varphi(q_k) + \frac{q_{k,s} - q_k}{q_k} \| q_k \theta \| \right)^s
 \left[ 1 - \left( \frac{q_k\varphi(q_k) + (q_{k,s} - q_k-1)\| q_k\theta\|}{q_k \varphi(q_k) + (q_{k,s}  - q_k) \| q_k\theta\|} \right)^s \right] \\
&= q_k \left( \varphi(q_k) + \frac{q_{k,s} - q_k}{q_k} \| q_k \theta \| \right)^s
 \left[ 1 - \left( 1 - \frac{\| q_k\theta\|}{q_k \varphi(q_k) + (q_{k,s} - q_k) \| q_k\theta\|} \right)^s \right] \\
&> q_k \left( \varphi(q_k) + \frac{q_{k,s} - q_k}{q_k} \| q_k \theta \| \right)^s \cdot \frac{s \| q_k\theta\|}{q_k \varphi(q_k) + (q_{k,s}  - q_k) \| q_k\theta\|} \\
&= s \| q_k\theta\| \left( \varphi(q_k) + \frac{q_{k,s} - q_k}{q_k} \| q_k \theta \| \right)^{-1+s}.
\end{split}
\end{equation*}
Here the second inequality is from the fact that $(1-t)^s < 1-st$ for $0 <s <1$ and $0 < t <1$. 
It is clear from \eqref{casei} that $q_{k_i,s} > q_{k_i}$.
Therefore, combining the claim \eqref{casei} and  the estimation on $\varphi(q_{k,s})$, we get
\begin{align*}
\varphi(m_i)^s &= \varphi(q_{k_i,s})^s  > s \| q_{k_i}\theta\|\left( \varphi(q_{k_i}) + \frac{q_{k_i,s} - q_{k_i}}{q_{k_i}} \| q_{k_i} \theta \| \right)^{-1+s}  \\
&> s \| q_{k_i}\theta\| \left(  \frac{2(q_{k_i,s} - q_{k_i})}{q_{k_i}} \| q_{k_i} \theta \| \right)^{-1+s}
> s \| q_{k_i}\theta\| \left(  \frac{2q_{k_i,s}}{q_{k_i}} \| q_{k_i} \theta \| \right)^{-1+s},
\end{align*}  
which is followed by  
$$\frac{1}{q_{k_i}^{K_i}} = \varphi(m_i) > s^{\frac 1s} \| q_{k_i}\theta\|^{\frac 1s} \left( \frac{2 m_i}{q_{k_i}} \| q_{k_i} \theta \| \right)^{- \frac1s+1} = 2 \left( \frac{s}{2} \right)^{\frac 1s} \frac{1}{q_{k_i}^{B_i} } \left( \frac{1}{q_{k_i}^{N_i-1}}\right)^{\frac 1{s}-1}.
$$
Thus, for large $i$, we get
$$
K_i-\varepsilon < B_i + (N_i-1)\left( \frac 1s -1 \right) =  1 + B_i - N_i +  \frac{1}{s}(N_i-1). 
$$
Applying \eqref{ineq3} we get
$$
K_i - \varepsilon <  \frac{1}{s - \varepsilon} +  \frac{1}{s} (N_i-1),   
$$
thus   
\begin{equation*}\label{bound2}
\frac {K_i} {N_i} < \frac 1s + \frac 1 {N_i} \left(\varepsilon + \frac 1 {s-\varepsilon} - \frac 1s \right) \le \varepsilon + \frac 1 {s-\varepsilon} < \frac{1}{s-\varepsilon(1+s^2)},
\end{equation*} 
where we apply $N_i \geq 1$ for the second inequality.
We conclude that
\begin{equation}\label{ineq4}
\frac{K_i}{N_i} < \frac{1}{s - \varepsilon (1+s^2)} .
\end{equation}

By \eqref{ineq3} and \eqref{ineq4}, from Lemma~\ref{lem:LR3} we have 
$$\hdim( F ) \ge s - \varepsilon(1+s^2). $$

\bigskip

Proof of Case (iii) : 

\smallskip

For each $\varepsilon >0$ we define a subsequence $\{q^*_k\}$ by
$$
q^*_k = \begin{cases}
q_{k,s} &\text{ if } \Lambda_k = \emptyset ,\\
\max \Lambda_k &\text{ if } \Lambda_k \ne \emptyset, \end{cases}
$$
where $$\Lambda_k = \left\{ q_{k,s} < n < q_{k+1} : \varphi(n) \ge \frac 1{n^{1/(s-\varepsilon)}} \right\}.$$
Then we have $q_{k,s} \le q^*_{k} < q_{k+1}$,
\begin{equation}\label{f2} \varphi(n) < \frac 1{n^{1/(s-\varepsilon)}} \ \text{ for } \ q^*_{k} < n < q_{k+1}\end{equation}
and
\begin{equation}\label{f3} \varphi(q^*_{k}) \ge \frac 1{ (q^*_{k})^{1/(s-\varepsilon)} } \ \text{ if } \ q^*_k > q_{k,s}.\end{equation}

By the minimality definition of $q_{k,s}$,
$$ \sum_{k=0}^\infty \left[ q_k \left( \varphi(q_k) + \frac{q^*_{k}-q_k}{q_k}\|q_{k}\theta\| \right)^s + \sum_{n = q^*_{k}+1}^{q_{k+1}-1} \varphi(n)^s \right] = \infty . $$
Since
$$ \sum_{k=0}^\infty \left(  \sum_{n = q^*_{k}+1}^{q_{k+1}-1} \varphi(n)^s \right)
< \sum_{k=0}^\infty \left( \sum_{n = q^*_{k}+1}^{q_{k+1}-1} \frac 1{n^{1 + \frac{\varepsilon}{s-\varepsilon}}} \right) < \sum_{n=1}^\infty \frac 1{n^{1 + \frac{\varepsilon}{s-\varepsilon}}} < \infty , $$
we get
\begin{equation*}
\sum_{k=0}^\infty q_k \left( \varphi(q_k) + \frac{q^*_{k}-q_k}{q_k}\|q_{k}\theta\| \right)^s = \infty .
\end{equation*}
Since
$$\sum_{k=0}^\infty q_k \left( \varphi(q_k) + \frac{q_{k,s}-q_k}{q_k}\|q_{k}\theta\| \right)^s < \infty ,$$
we get
\begin{align*}
\sum_{k=0}^\infty q_k \left(  \frac{q^*_{k} - q_{k,s}}{q_k}\|q_{k}\theta\| \right)^s
&= \infty .
\end{align*}
Therefore, we conclude that for each $\varepsilon >0$ there exists a subsequence $\{k_i\}$ such that
$$ \frac{q^*_{k_i} - q_{k_i,s}}{q_{k_i}} \| q_{k_i}\theta \| >\left( \frac{1}{q_{k_i}} \right)^{1/(s-\varepsilon)}.$$
Thus
\begin{equation}\label{f4} \frac{q^*_{k_i}}{q_{k_i}} \| q_{k_i}\theta \| >\left( \frac{1}{q_{k_i}} \right)^{1/(s-\varepsilon)}
\ \text{ and } \  q^*_{k_i} > q_{k_i,s}. \end{equation}

Put
$$ m_i := q^*_{k_i} = q_{k_i}^{N_i}, \qquad \frac{1}{q_{k_i}^{B_i}} := \| q_{k_i} \theta \|.$$ 
On one hand, by \eqref{f4},
$$ q_{k_i} \left( \frac{q^*_{k_i}}{q_{k_i}} \| q_{k_i}\theta \| \right)^{s-\varepsilon} = q_{k_i}^{1- (s-\varepsilon)(1+B_i-N_i) } > 1, $$
which implies
\begin{equation}\label{ineq6}
1+B_i-N_i < \frac{1}{s-\varepsilon}.
\end{equation}

On the other hand, since $q^*_{k_i} > q_{k_i,s}$ (see \eqref{f4}), by \eqref{f3}, one has that
$$\varphi(m_i)=\varphi(q^*_{k_i})
\ge \left( \frac 1{q^*_{k_i}}\right)^{\frac{1}{s-\varepsilon}} = \left( \frac 1{q_{k_i}^{N_i}}\right)^{\frac{1}{s-\varepsilon}}. $$
Set
\begin{equation}\label{ineq5}
K_i : =\frac{N_i}{s-\varepsilon}>1 
\end{equation} and define
$$G_i := \left \{ y : \| n\alpha - y \| < \frac 1{2q_{k_i}^{K_i}} \text{ for some } m_{i-1} < n \le m_i \right \}. $$
It is clear that
$$ G := \bigcap_{i =1}^\infty G_i \subset E_\varphi(\theta).$$

Therefore, 
by \eqref{ineq6}, \eqref{ineq5} and Lemma~ \ref{lem:LR3}
\begin{equation*}
\hdim (E_\varphi(\theta)) \ge \hdim( G ) \ge s-\varepsilon.  \qedhere
\end{equation*} 
\end{proof}

\section{Bounds for the dimension}\label{sec:relations}

In this section, we show that the previous known dimensional results are deduced by the main theorem.

\begin{corollary}[Schmeling-Troubetzkoy \cite{ScTr}, Bugeaud \cite{Bu}]\label{C1}
$$  l_\varphi \le  \hdim (E) \le u_\varphi. $$
\end{corollary}

\begin{proof}
(i) By the definition of $l_\varphi$, for any $\varepsilon >0$ for sufficiently large $n$
$$\varphi(n) > \frac{1}{n^{1/(l-\varepsilon)}}.$$
Therefore, we have for some large $k_0$
\begin{align*}
\sum_{k=0}^\infty \left[ q_k \left( \varphi(q_k) + \frac{q_{k,l-\varepsilon}-q_k}{q_k}\|q_{k}\theta\| \right)^{l-\varepsilon} + \sum_{n = q_{k,l-\varepsilon}+1}^{q_{k+1}-1} \varphi(n)^{l-\varepsilon} \right]
&\ge \sum_{k=k_0}^\infty q_k \varphi(q_k)^{l-\varepsilon}\\
&> \sum_{k=0}^\infty 1 = \infty,
\end{align*}
Hence, we have  $$l_\varphi \le \hdim (E).$$

(ii) By the definition of $u_\varphi$, for any $\rho >1$ for sufficiently large $n$
$$\varphi(n) < \frac{1}{n^{1/\rho u}}.$$
Therefore, by the minimality of the definition of $q_{k,\rho^2 u}$, we have
\begin{multline*}
\sum_{k=0}^\infty \left[ q_k \left( \varphi(q_k) + \frac{q_{k, \rho^2 u}-q_k}{q_k}\|q_{k}\theta\| \right)^{\rho^2 u} + \sum_{n = q_{k,\rho^2 u}+1}^{q_{k+1}-1} \varphi(n)^{\rho^2 u} \right] \\
\le\sum_{k=0}^\infty \left( q_k \varphi(q_k)^{\rho^2 u} + \sum_{n = q_{k}+1}^{q_{k+1}-1}  \varphi(n)^{\rho^2 u}  \right)
\le \sum_{k=0}^\infty \left( \frac{1}{q_k^{\rho -1}} + \sum_{n = q_k}^{q_{k+1}-1} \frac{1}{n^{\rho}} \right) < \infty,
\end{multline*}
which implies that
\begin{equation*}
\hdim (E) \le \rho u_\varphi. \qedhere
\end{equation*}
\end{proof}

\begin{corollary}[Xu \cite{Xu}]
$$ \hdim (E) \ge \limsup_{k \to \infty} \frac{\log q_k}{-\log \varphi(q_k)}.$$
\end{corollary}

\begin{proof}
Let
$$t : = \limsup_{k \to \infty} \frac{\log q_k}{-\log \varphi(q_k)}.$$
Then for any $\varepsilon >0$ there exists $q_{k_i}$ such that $$ \varphi(q_{k_i}) > \frac{1}{q_{k_i}^{1/(t-\varepsilon)}}.$$
Hence, we have
\begin{align*}
\sum_{k=0}^{\infty} \left( q_{k} \left( \varphi(q_{k}) + \frac{q_{k,t-\varepsilon}-q_{k}}{q_{k}}\|q_{k}\theta\| \right)^{t-\varepsilon} + \sum_{n = q_{k,t-\varepsilon}+1}^{q_{{k}+1}-1} \varphi(n)^{t-\varepsilon} \right) &\ge \sum_{i=1}^\infty q_{k_i} \varphi(q_{k_i})^{t-\varepsilon} \\
&> \sum_{i=1}^\infty 1= \infty,
\end{align*}
which implies that
\begin{equation*}
 t \le  \hdim (E).  \qedhere
\end{equation*}
\end{proof}

\begin{corollary}[Liao and Rams \cite{LR}]
If $\frac{1 + u_\varphi}{1 + w} < u_\varphi$, then
$$ \hdim (E) \ge \frac{1 + u_\varphi}{1 + w}. $$ 
\end{corollary}

\begin{proof} Write $u_{\varphi}=u$.
Let $$z := \frac{1 + u}{1 + w} < u$$
 and $0 < \varepsilon < z$ be given.
Since $\limsup \frac{- \log \| q_k \theta \|}{\log q_k} = w$,
for  all $n$ large, $$\| q_{n} \theta \| > \frac{1}{q_{n}^{w + \varepsilon}} .$$

There are two cases: \\
(a) If there are infinitely many $n$ such that $q_k \le n \le q_{k,z-\varepsilon}$
and
$$ \varphi(n) > \frac{1}{n^{1/(u-\varepsilon)}}.$$
(b) Suppose that for all large $n$ such that  $q_k \le n \le q_{k,z-\varepsilon}$
$$ \varphi(n) \le \frac{1}{n^{1/(u-\varepsilon)}}.$$

Case (a):
\smallskip

There exist $\{n_i\}$  and $\{k_i\}$ such that
$$ \varphi(n_i) > \frac{1}{n_i^{1/(u-\varepsilon)}},  \qquad  q_{k_i} \le n_i \le q_{k_i,z-\varepsilon}.$$
Then we have
$$ \varphi(q_{k_i}) \ge \varphi(n_i) > \frac{1}{n_i^{1/(u-\varepsilon)}} \ge \frac{1}{(q_{k_i,z-\varepsilon})^{1/(u-\varepsilon)}}.$$
Thus, since $q_{k_i,z-\varepsilon} \ge q_{k_i}$, we get
\begin{equation*}
\varphi(q_{k_i}) + \frac{q_{k_i,z-\varepsilon}}{q_{k_i}}\|q_{k_i}\theta\|
> \frac{1}{(q_{k_i,z-\varepsilon})^{1/(u-\varepsilon)}} + \frac{q_{k_i,z-\varepsilon}}{q_{k_i}} \|q_{k_i}\theta\|
\ge \min_{t \ge q_{k_i}} f_{u-\varepsilon}(t),
\end{equation*} where we set $f_v(t):=\frac{1}{t^{1/v}} + t \frac{\| q_{k_i} \theta \|}{q_{k_i}}$.

By elementary calculus, we have
$$
\min_{t \ge t_0} f_v(t) = \begin{cases}
f_v \left( \left(\frac{q_{k_i}}{v \| q_{k_i} \theta \|}\right)^{v/(1+v)} \right)
&\text{ if } 0 < t_0 \le \left(\frac{q_{k_i}}{v \| q_{k_i} \theta \|}\right)^{v/(1+v)}, \\
f_v(t_0) &\text{ if }  t_0 > \left(\frac{q_{k_i}}{v \| q_{k_i} \theta \|}\right)^{v/(1+v)}.
\end{cases}
$$
Therefore, if $q_{k_i} \le \left(\frac{q_{k_i}}{(u-\varepsilon)\| q_{k_i} \theta \|} \right)^{(u-\varepsilon)/(1+u-\varepsilon)}$, then
\begin{align*}
\varphi(q_{k_i}) + \frac{q_{k_i,z-\varepsilon}}{q_{k_i}}\|q_{k_i}\theta\|
&> \left( \frac{ (u-\varepsilon) \| q_{k_i} \theta \|}{q_{k_i}} \right)^{\frac{1}{1+u-\varepsilon}}  +  \left(\frac{q_{k_i}}{(u-\varepsilon)\| q_{k_i} \theta \|} \right)^{\frac{u-\varepsilon}{1+u-\varepsilon}}\frac{\|q_{k_i}\theta\|}{q_{k_i}} \\
&> (u-\varepsilon)^{\frac{1}{1+u-\varepsilon}}\left(\frac{1}{q_{k_i}} \right)^{\frac{1+w+\varepsilon}{1+u-\varepsilon}}+ \| q_{k_i} \theta \|  \\
&> \left(\frac{u-\varepsilon}{q_{k_i}} \right)^{\frac{1}{z-\varepsilon}} + \| q_{k_i} \theta \|
\end{align*}
and if $q_{k_i} > \left(\frac{q_{k_i}}{(u-\varepsilon)\| q_{k_i} \theta \|}\right)^{(u-\varepsilon)/(1+u-\varepsilon)}$, then
\begin{align*}
\varphi(q_{k_i}) + \frac{q_{k_i,z-\varepsilon}}{q_{k_i}}\|q_{k_i}\theta\|
&> \frac{1}{(q_{k_i})^{1/(u-\varepsilon)}} + \| q_{k_i} \theta \| \ge \frac{1}{(q_{k_i})^{1/(z-\varepsilon)}} + \| q_{k_i} \theta \|.
\end{align*}
Thus,
\begin{equation*}
q_{k_i} \left(\varphi(q_{k_i}) +\frac{q_{k_i,z-\varepsilon}-q_{k_i}}{q_{k_i}}\|q_{k_i}\theta\| \right)^{z-\varepsilon}> u - \varepsilon .
\end{equation*}
Hence, we have
$$
\sum_{i=1}^{\infty} \left[ q_{k_i} \left( \varphi(q_{k_i}) + \frac{q_{k_i,z-\varepsilon} - q_{k_i}}{q_{k_i}}\|q_{k_i}\theta\| \right)^{z-\varepsilon} + \sum_{n = q_{k_i,z-\varepsilon} +1}^{q_{{k_i}+1}-1} \varphi(n)^{z-\varepsilon} \right] = \infty,
$$
which implies that
$$z -\varepsilon \le \hdim (E) .$$
\medskip

Case (b):
\smallskip

By the definition of $u = \limsup\frac{\log n}{-\log \varphi(n)}$,
there exist $\{n_i\}$ and $\{k_i\}$ such that
$$ \varphi(n_i) > \frac{1}{n_i^{1/(u-\varepsilon/2)}}, \qquad  q_{k_i,z-\varepsilon} < n_i < q_{k_i+1}.$$
Then for large $i$ we get
$$ \frac{1}{ ( q_{k_i,z-\varepsilon} )^{1/(u - \varepsilon)}} \ge \varphi(q_{k_i,z-\varepsilon}) \ge \varphi(n_i) > \frac{1}{(n_i)^{1/(u-\varepsilon/2)}}.$$
Thus, for large $i$ we have
$$ q_{k_i,z-\varepsilon} < (n_i)^{(u - \varepsilon)/(u-\varepsilon/2)} < \frac{n_i}{2}.$$
Therefore, we have for large $i$
\begin{align*}
\sum_{n = q_{k_i,z-\varepsilon}+1}^{q_{{k_i}+1}-1} \varphi(n)^{z-\varepsilon}
&\ge \sum_{n = q_{k_i,z-\varepsilon}+1}^{q_{{k_i}+1}-1} \varphi(n)^{u-\varepsilon}
\ge \sum_{n = q_{k_i,z-\varepsilon}+1}^{n_i} \varphi(n_i)^{u-\varepsilon} \\
&\ge \frac{n_i- q_{k_i,z-\varepsilon}}{ (n_i)^{(u - \varepsilon)/(u-\varepsilon/2)} } > 1 .
\end{align*}
Hence, we have
$$\sum_{k=0}^\infty \left[ q_k \left( \varphi(q_k) + \frac{q_{k,z-\varepsilon}-q_k}{q_k}\|q_{k}\theta\| \right)^{z-\varepsilon} + \sum_{n = q_{k,z-\varepsilon}+1}^{q_{k+1}-1} \varphi(n)^{z-\varepsilon} \right] = \infty, $$
which implies that
\begin{equation*}
z - \varepsilon \le  \hdim (E) . \qedhere
\end{equation*}
\end{proof}

\section*{Acknowledgements}
The authors would like to thank Huazhong University of Science and Technology for inviting two of the authors.
D.K. is supported by NRF-2015R1A2A2A01007090.
M.R. is supported by National Science Centre grant 2014/13/B/ST1/01033 (Poland). B.W. is supported by NSFC 11471130 and NCET-13-0236.

\end{document}